\documentclass[11pt]{amsart}
\usepackage{color}
\usepackage{enumerate}

\newtheorem{teo}{Theorem}[section]

\newtheorem{prop}[teo]{Proposition}
\newtheorem{cor}[teo]{Corollary}

\theoremstyle{definition}
\newtheorem{definiz}[teo]{Definition}
\newtheorem{example}[teo]{Example}
\newtheorem{remark}[teo]{Remark}

\newcommand{\R}{\mathbb R}

\newcommand{\Z}{\mathbb Z}
\newcommand{\C}{\mathbb C}

\newcommand{\So}{\mathcal{SR}}

\newcommand{\SF}{\mathbb S}

\DeclareMathOperator{\pv}{\wedge \mspace{-9.5mu}_* \ }

\newcommand{\N}{\mathbb N}

\newcommand{\HH}{\mathbb H}

\theoremstyle{remark} 

\numberwithin{equation}{section}

\begin{document}

\title[$*$-exponential of slice-regular functions]
{$*$-exponential of slice-regular functions}

\author[A. Altavilla]{A. Altavilla${}^{\dagger,\ddagger}$}\address{Altavilla Amedeo: Dipartimento Di Matematica, Universit\`a di Roma "Tor Vergata", Via Della Ricerca Scientifica 1, 00133, Roma, Italy}\email{altavilla@mat.uniroma2.it}

\author[C. de Fabritiis]{C. de Fabritiis${}^{\dagger}$}\address{Chiara de Fabritiis: Dipartimento di Ingegneria Industriale e Scienze
Matematiche, Universit\`a Politecnica delle Marche, Via Brecce Bianche, 60131,
Ancona, Italy}\email{ fabritiis@dipmat.univpm.it}
\thanks{${}^{\dagger}$GNSAGA of INdAM, ${}^{\ddagger}$FIRB 2012 {\sl Geometria differenziale e teoria geometrica delle funzioni}, SIR grant {\sl ``NEWHOLITE - New methods in holomorphic iteration''} n. RBSI14CFME and SIR grant {\sl AnHyC - Analytic aspects in complex and hypercomplex geometry} n. RBSI14DYEB.}

\date{\today }


\begin{abstract} 
According to~\cite{C-S-St-2} we define the $*$-exponential of a slice-regular function, which can be seen as a generalization of the complex exponential to quaternions.   
Explicit formulas for $\exp_*(f)$ are provided, also in terms of suitable sine and cosine functions. 
We completely classify under which conditions the $*$-exponential of a function is either slice-preserving or $\C_J$-preserving for some $J\in\SF$ and show that $\exp_*(f)$ is never-vanishing. 
Sharp necessary and sufficient conditions are given in order that $\exp_*(f+g)=\exp_*(f)*\exp_*(g)$, finding an exceptional and unexpected case in which equality holds even if $f$ and $g$ do not commute.
We also discuss the existence of a square root of a slice-preserving regular function, characterizing slice-preserving functions (defined on the circularization of simply connected domains) which admit square roots. 
Square roots of this kind of functions are used to provide a further formula for $\exp_{*}(f)$.
A number of examples is given throughout the paper.
\end{abstract}
\maketitle

\section{Introduction}

In classical complex analysis the exponential map has a tremendous role in the study of growth of holomorphic functions, differential equations and uniformization theorem. In this paper we investigate the behaviour of a quaternionic  analogous of the complex exponential map. 
In particular, special features of the exponential in the complex case include the fact that it never vanishes, the expression of the exponential of a pure imaginary in terms of sine and cosine, and the formula which gives the exponential of a sum as the product of the exponentials of the two summands.
These properties will be the object of our study in the setting of slice-regularity.

We briefly introduce quaternions and quaternionic slice-regular functions. 
Given an alternating triple $i,j,k$ with $i^2=j^2=k^2=-1$ and $k=ij=-ji$, we denote by $\HH$ the real algebra of quaternions 
$$\HH=\{q=q_0+q_1i+q_2j+q_3k\ :\ q_0,q_1,q_2,q_3\in \R\}.$$ The conjugation on $\HH$ is given by $q^c=q_0-(q_1i+q_2j+q_3k)$ and
we sometimes write ${\rm Re} (q)=q_0=\frac12(q+q^c)$ and 
$\vec q=q- {\rm Re} (q)$, so that $q=q_0+\vec q$; the quaternion $\vec q$ is often denoted by ${\rm Im}(q)$.
 
The following subsets of $\HH$ have special interest:
\begin{align*}
\SF&=\{q\in\HH\ : \ q^2=-1\}=\{q_1i+q_2j+q_3k\ : \ q_1^2+q_2^2+q_3^2=1\}\\
{\rm Im}(\HH)&=\{q\in\HH: {\rm Re} (q)=0\}=\bigcup_{I\in \SF} \, \R I;
\end{align*}
the first can be identified with the standard sphere of $\R^{3}\simeq {\rm Span}_{\R}(i,j,k)$ and the second with $\R^3\simeq {\rm Span}_{\R}(i,j,k)$. For any $I\in\SF$ we set  $\C_I=
{\rm Span}_{\R}(1,I)$. 

Given any set $D\subset \C=\{\alpha+\imath \beta\,|\,\alpha,\beta\in\R\}$, we define its \textit{circularization} as 
$$
\Omega_D=\{\alpha+\beta I\in\HH\, | \, \alpha+\imath \beta\in D, I\in\SF\}\subseteq \HH.
$$
Subsets of $\HH$ of this form will be called \textit{circular sets}. If $D$ is a singleton $\{q\}$ then we also denote the sphere $\Omega_{\{q\}}$ with real center ${\rm Re}(q)$ by $\SF_{q}$. Notice that, if $D$ is open in $\C$, then $\Omega_{D}$ is open in $\HH$. Here we assume that $\HH\simeq \R^{4}$ is equipped with the Euclidean topology. In order to simplify the notation we will often drop the subscript $D$ writing $\Omega$ in place of $\Omega_{D}$.

The following definitions identify the functions we will work with (see~\cite{G-P})

\begin{definiz}
Let $D\subset \C$ be invariant under conjugation.  
A  function $F:D\to\HH\otimes_{\R} \C$ is a {\sl stem} function if $F(\overline z)=\overline{F(z)}$ where $\overline{p+\imath q}=p-\imath q$ for any $p+\imath q\in\HH\otimes_{\R} \C$.
A {\sl slice} function $f:\Omega_D\to \HH$ is a function induced by a stem function $F=F_1+\imath F_2:D\to\HH\otimes_{\R} \C$ in the following way:
$f(\alpha+\beta I)=F_1(\alpha+\imath \beta)+I F_2(\alpha+\imath \beta)$. Such a function will also be denoted by $f=\mathcal I(F)$.
\end{definiz}

Notice that a slice function $f$ is induced by a unique stem function $F = F_{1} + \imath F_{2}$, given by 
$F_{1}(\alpha+\imath \beta) = \frac{1}{2}(f(\alpha+\beta I)+f(\alpha-\beta I))$ and $F_{2}(\alpha+\imath \beta) = -\frac{1}{2}I(f(\alpha+\beta I)-f(\alpha-\beta I))$ for
any $I \in \SF$.

\begin{definiz}
Let $\Omega\subset \HH$ be a circular domain, that is a circular connected open subset of $\HH$.
A {\sl slice} function $f=\mathcal I(F):\Omega\to \HH$ is slice {\sl regular} if $F$ is holomorphic with respect to the natural complex structures of $\C$ and $\HH\otimes_{\R} \C$. We denote by $\So(\Omega)$ the set of all slice-regular functions on $\Omega$.
\end{definiz}
Notice that $\So(\Omega)$ has a natural structure of right $\HH$-module.\\
In what follows we always assume $\Omega$ is a (non-empty) circular domain of $\HH$.
If $\Omega$ contains real points, a function is slice-regular if and only if it is regular in the sense of Cullen (see~\cite{G-S-St}). A useful result for slice-regular functions is the following

\begin{prop}[Representation formula] Let $f\in\So(\Omega)$ and let $\alpha+\beta J\in\Omega$. For all $I\in\SF$ we have 
$$
f(\alpha+\beta J)=\frac{1-JI}2f(\alpha+\beta I)+\frac{1+JI}2f(\alpha-\beta I). 
$$
\end{prop}
Therefore if
$f_I:\Omega_I=\Omega\cap \C_I\to \HH$  is a holomorphic function with respect to the left multiplication by $I$, then there exists a unique slice-regular function $g:\Omega\to\HH$ such that $f_I=g|_{\Omega_I}$. Such a function will be called the {\sl regular extension of $f_I$}, (see~\cite{G-S-St}, p 9).

In general the pointwise product of two slice-regular functions is no more slice. Nonetheless, this problem can be overcome by defining the following non-commutative product (see~\cite{C-G-S-St} also).
\begin{definiz}\label{prodotto-stella}
Let $f=\mathcal I(F)$ and $g=\mathcal I(G)$ be two slice functions on $\Omega$. We denote by $f*g$ their {\sl $*$-product} defined by $f*g=\mathcal I(FG)$ where $FG$ is the pointwise product with values in $\HH\otimes_{\R} \C$, {\sl i.e.} 
$(p+\imath q)(p'+\imath q')=pp'-qq'+\imath (pq'+qp')$.
\end{definiz}

In $\So(\Omega)$ we can also define the {\sl conjugate function} of $f$ that is 
$f^c=\mathcal I(F_1^c+\imath F_2^c)$ if $f=\mathcal I(F_1+\imath F_2)$ and the functions $F_1^c$ and $F_2^c$ are obtained from $F_1$ and $F_2$ by pointwise conjugation in $\HH$.
Given $f\in \So(\Omega)$ its {\sl symmetrized function} $f^s$ is given by $f^s=f^c*f=f*f^c$.

In~\cite{G-P} the $*$-product $f*g$ and the symmetrized function $f^{s}$ of $f$ are called slice product and normal function of $f$, and are denoted by $f\cdot g$ and $N(f)$, respectively.

It is known that if $\alpha+\beta I$ is such that $f(\alpha+\beta I)=0$, then any point in the set $\SF_{\alpha+\beta I}=\{\alpha+\beta J\,|\,  J\in\SF\} $ is a zero for $f^s$. In particular if $f^s$ is never-vanishing then also $f$ is. Moreover, the zero set of a slice-regular function $f\not\equiv0$ is closed with empty interior and, if $f^s\not\equiv0$, it is a union of isolated points and isolated spheres of the form  $\SF_{\alpha+\beta J}$ for suitable $\alpha,\beta\in\R$  (see~\cite{G-P, G-S}).

We now introduce two special classes of slice-regular functions. 
A function $f=\mathcal I(F_1+\imath F_2)\in\So(\Omega)$ is {\sl slice-preserving} if both $F_1$ and $F_2$ are real-valued. The class of such functions will be denoted by  $\So_\R(\Omega)$
(they are called $\HH$-intrinsic in \cite{G-M-P} and quaternionic-intrinsic in \cite{C-S-St-2}).
We remark that for any $f\in\So(\Omega)$ its symmetrized function $f^s$ belongs to $\So_\R(\Omega)$.
 
Given $J\in\SF$, a function $f=\mathcal I(F_1+\imath F_2)\in\So(\Omega)$ is said to be
  {\sl $\C_J$-preserving} if both $F_1$ and $F_2$ are $\C_J$-valued;  the class of such functions will be denoted by $\So_J(\Omega)$ (see~\cite{G-M-P}).

For slice-preserving and $\C_J$-preserving functions, the $*$-product has special features:
\begin{itemize}
\item for any $f\in\So_\R(\Omega)$ and $g\in\So(\Omega)$ we have $f*g=g*f=fg$ (that is, the $*$-product $f*g$ coincides with the pointwise product $fg$);
\item choosen $J\in\SF$, for any $f,g\in\So_J(\Omega)$ we have $f*g=g*f$ (and $f*g$ coincides with the regular extension of the pointwise product of $f_J$ and $g_J$);
\item if $\rho_1,\rho_2\in \So_\R(\Omega)$ and $a_1,a_2\in\HH$ then $(\rho_1a_1)*(\rho_2a_2)=\rho_1\rho_2a_1a_2$.
\end{itemize}

We have now the main tools to define the $*$-exponential of a slice-regular function.
The first definition is the natural extension of the complex exponential to $\HH$ and indeed its restriction to any complex line $\C_I$, $I\in\SF$, coincides with the complex exponential. 

The function $\exp:\HH\to\HH$ is given by 
$\exp q=e^q=\sum_{n\in\N}\frac{q^n}{n!}$. Trivially $\exp\in\So_\R(\HH)$. 

Given $f\in\So(\Omega)$ the composition $\exp\circ f$ is not always slice-regular. In~\cite{C-S-St-2}, Colombo, Sabadini and Struppa gave the following definition which coincides with $\exp\circ f$ if $f\in\So_\R(\Omega)$ (see also~\cite{R-W} and~\cite{V}, where several different regular compositions are introduced). 
\begin{definiz} If $f\in\So(\Omega)$ the $*$-exponential of $f$ is defined as  
$\exp_*(f)=\sum_{n\in\N}\frac{f^{*n}}{n!}$.
\end{definiz}

Let $f\in\So(\Omega)$ and let $n$ be a positive integer. Notice that, given any (non-empty) circular compact subset $K$ of $\Omega$, $\max_{q\in K} |f^{*n}(q)|\leq\max_{q\in K}|f^{n}(q)|$. This follows easily by induction on $n$ observing that $f^{*n}(q) = 0$ if $f(q) = 0$ and $|f^{*n}(q)| = |f(q)| \cdot |f^{*(n-1)}(f(q)^{-1}qf(q))|$
if $f(q)\neq 0$ (see~\cite[Theorem 3.4]{G-S-St}). Thanks to these estimates, the above series converges uniformly on compact subsets of $\Omega$ and hence $\exp_{*}(f)$ is a well defined function in $\So(\Omega)$.

The function $\exp_{*}(f)$ will be the main object of our study, which will be organized as follows.

Section 2 contains a new interpretation of the $*$-product given by Formula~(\ref{star-formula}). This result allows us to give a necessary and sufficient condition for the commutation of two functions with respect to this product. 
We will exploit Formula~(\ref{star-formula}) extensively in Section 4, where it will be a new, useful tool to simplify calculations.

In Section 3 we investigate on the possibility of finding a square root of a slice-preserving regular function in the space $\So_\R(\Omega)$, when each connected component of $\Omega_I=\C_I\cap \Omega$ is simply connected, obtaining a complete answer to this question. Notice that, since $\Omega$ is connected by assumption, the set $\Omega_I$ has at most two connected components. 

A part of the necessary and sufficient condition in order that a function in $\So_\R(\Omega)\setminus\{0\}$ has a square root in $\So_\R(\Omega)$ is given in terms of the spherical and isolated multiplicities of its zeroes (see~\cite{G-S-St}, Paragraph 3.6, and Definition 3.37 in particular).
The result is the following.
\begin{prop}
Suppose each connected component of $\Omega_I$ is simply connected for some (and hence any) $I\in\SF$.
Given $h\in \So_\R(\Omega)\setminus\{0\}$ there exists $f\in \So_\R(\Omega)$ such that $f^2=h$ iff
\begin{enumerate}[(i)]
\item the zero set of $h$ consists of real zeroes of even isolated multiplicity and isolated spheres $\SF_{q}$ with spherical multiplicity multiple of $4$;
\item if $\Omega\cap \R\neq\emptyset$ then $h(\Omega\cap \R)\subseteq [0,+\infty)$.
\end{enumerate}
\end{prop}

In particular we apply this result when $h=g^s$ for $g\in\So(\Omega)$.

Last section is devoted to the study of the $*$-exponential of a slice-regular function.
We first give a sufficient condition, namely commutativity of the $*$-product of $f$ and $g$, for the equality
$\exp_*(f+g)=\exp_*(f)*\exp_*(g)$ to hold. Given $f=f_0+f_v$ (see Definition~\ref{vector} and subsequent remark for the notation), we then find an explicit expression for $\exp_*(f)$ in terms of $f_0$ and $f_v$.
In Corollaries~\ref{sine-cosine} and~\ref{true-sine} 
this allows us to rewrite $\exp_*(f)$ in terms of suitable sine and cosine functions. 
Moreover we completely classify under which conditions the $*$-exponential of a function is either in $\So_\R(\Omega)$ or in $\So_J(\Omega)$ for some $J\in\SF$.

In Proposition~\ref{real-part} we compute the symmetrized function of the $*$-exponential of $f$ in terms of the real part of $f$, proving as a consequence that $\exp_*(f)$ is never-vanishing. 
Finally, in Theorem~\ref{super-equality} we give necessary and sufficient conditions on $f$ and $g$ in order that $$\exp_*(f+g)=\exp_*(f)*\exp_*(g).$$ 
Quite surprisingly we are able to find a large bunch of slice regular functions which do not commute under the $*$-product and for which nonetheless the above equality holds.
This last result is followed by four examples which illustrate the sharpness of the required conditions.

The authors warmly thank the anonymous referee for many suggestions which helped to improved the quality of the paper.

\section{regular product: a further interpretation}

The following result, which is due to Colombo, Gonzales-Cervantes and Sabadini (see~\cite{C-GC-S}, Proposition~3.12), gives a way to decompose a given slice-regular function by means of four slice-preserving regular functions (see also \cite{G-M-P}, Lemma 6.11).

\begin{prop}\label{GMP} 
Let $\{1,I,J,K\}$ be a vector basis of $\HH$. Then the map 
$$
\left(\So_\R(\Omega)\right)^4\ni (f_0,f_1,f_2,f_3)\mapsto f_0+f_1I+f_2J+f_3K\in \So(\Omega)
$$
is bijective. In particular it follows that given any $f\in \So(\Omega)$ there exist and are unique $f_0,f_1,f_2,f_3\in\So_\R(\Omega)$ such that 
$$
f=f_0+f_1I+f_2J+f_3K.
$$
Moreover if $I\in\SF$ then $f\in \So_I(\Omega)$ iff $f_2\equiv f_3\equiv 0$. 
\end{prop}

\begin{remark}
If $I,J,K\in\SF$, then $(f_0,f_1,f_2,f_3)$ is the $4$-tuple associated to $f$  if and only if $(f_0,-f_1,-f_2,-f_3)$ is the $4$-tuple associated to $f^c$; in particular $f_0=\frac{f+f^c}2$ does not depend on the chosen 
vector basis $\{1,I,J,K\}$ of $\HH$, provided $I,J,K\in\SF$.

\end{remark}

\begin{definiz}\label{vector} 
Given $f\in\So(\Omega)$, we define the slice-regular function $f_v$  on $\Omega$ by
$f_v=\frac{f-f^c}2$.
\end{definiz}

\begin{remark}
Trivially $f_0+f_v=f$ and $f_0-f_v=f^c$. 
According to the notation of Proposition~\ref{GMP} we have $f_v=f_1I+f_2J+f_3K$. 
\end{remark}

The reason we choose the underscore $v$ for the above function is that, in a certain sense, it represents the ``vector'' part of the quaternionic-valued function $f$, according to the splitting $q=q_0+q_1i+q_2j+q_3k=q_0+\vec{q}$. 

The above result and definition allow us to describe the $*$-product of two slice-regular functions in terms of intrinsic operators.

Let now $I,J,K$ be an orthonormal basis of ${\rm Im}(\HH)$ with $IJ=K$.

\begin{definiz} 
Given $f,g\in\So(\Omega)$, we denote by $f\pv g$ 
and $\langle f, g\rangle_*$ the slice-regular functions given by 
$$
(f\pv  g)(q)=\frac{(f*g)(q)-(g*f)(q)}2,
\qquad 
\langle f, g\rangle_*(q)=(f*g^c)_0(q)
$$
for any $q\in\Omega$.
\end{definiz}

\begin{remark}
In terms of the notation of Proposition~\ref{GMP} we can rewrite the above intrinsic expressions in the following form:
\begin{align*}
f\pv  g &=(f_2g_3-f_3g_2)I+(f_3g_1-f_1g_3)J+(f_1g_2-f_2g_1)K \\
\langle f, g\rangle_*&=f_0g_0+f_1g_1+f_2g_2+f_3g_3.
\end{align*}
which in particular show that 
\begin{align*}
f\pv  g &=f_v\pv  g_v=-g_v\pv  f_v\\
\langle f, g\rangle_*&=\langle g, f\rangle_*=\langle f^c, g^c\rangle_*= f_0g_0-(f_v*g_v)_0.
\end{align*}
\end{remark}

These operators can be used to write the $*$-product of two slice-functions in a more explicit form which recalls the formula for the product of two quaternions in the scalar-vector form. This different expression of the $*$-product will turn out to be a useful tool in the study of slice regular functions, simplifying and speeding formerly long and complicated computations.

\begin{prop}\label{star-product}
The $*$ product of $f=f_0+f_v$ and $g=g_0+g_v$
is given by
\begin{equation}\label{star-formula}
f*g=f_0g_0-\langle f_v, g_v\rangle_*+f_0g_v+g_0f_v+f_v\pv  g_v
\end{equation}
\end{prop}
\begin{proof}
Choosing the standard vector basis $\{1,i,j,k\}$ of $\HH$ on $\R$ and writing $f=f_0+f_1i+f_2j+f_3k$ and $g=g_0+g_1i+g_2j+g_3k$ we have 
\begin{align*}
f*g&=(f_0+f_1i+f_2j+f_3k)*(g_0+g_1i+g_2j+g_3k)\\
&=
f_0g_0-f_1g_1-f_2g_2-f_3g_3+f_0(g_1i+g_2j+g_3k)+g_0(f_1i+f_2j+f_3k)\\
&+(f_2g_3-f_3g_2)i+(f_3g_1-f_1g_3)j+(f_1g_2-f_2g_1)k\\
&=f_0g_0-\langle f_v, g_v\rangle_*+f_0g_v+g_0f_v+f_v\pv  g_v,
\end{align*}
where the second equality is a consequence of the definition of the $*$-product and of the fact that $f_0,\ldots,f_3,g_0,\ldots,g_3$ belong to $\So_\R(\Omega)$.
\end{proof}

\begin{remark}\label{symmetrized}
Given $f\in\So(\Omega)$, we can recognize in $\langle f, f\rangle_*
$ its symmetrized function $f^s$.
Written $f$ as $f_0+f_1i+f_2j+f_3k$, 
since $f^c=f_0-f_v$
we have 
$$\langle f, f\rangle_*=f_0^2+f_1^2+f_2^2+f_3^2=f_0g_0-\langle f_v, -f_v\rangle_*+f_0(-f_v)+f_0f_v+f_v\pv  (-f_v)=f*f^c.$$
\end{remark}

\begin{remark}\label{null-symmetrized}
Notice that if $\Omega$ contains real points, then 
 $f_v^s\equiv0$ if and only if $f_v\equiv 0$. If $\Omega$ does not contain any real point, there exist examples in which $f_v^s\equiv 0$ and $f_v\not\equiv 0$.  
Indeed if $\Omega$ contains real points and 
 $f_v^s\equiv0$ then $f_1^2+f_2^2+f_3^2\equiv 0$. On real points  $f_1$, $f_2$ and $f_3$ take real values and therefore 
 $f_1^2+f_2^2+f_3^2\equiv 0$ implies $f_1=f_2=f_3=0$ on the intersection between $\Omega$ and $\R$, which entails $f_1\equiv f_2\equiv f_3\equiv0$ by the identity principle (see~\cite{G-S-St}, Theorem 1.12). If the domain $\Omega$ does not contain any real point there exist $f\in\So(\Omega)$ such that $f_v\not\equiv0$ and $f_v^s\equiv0$ (see~\cite{AA}, Example 2). 
\end{remark}

In a certain sense, Formula~\ref{star-formula} allows us to untangle the $*$-product on $\So(\Omega)$, confining the skewness to $\HH$ and thus simplifying its computaion with respect to Definition~\ref{prodotto-stella} or Theorem~3.4 in \cite{G-S-St}.
Next result characterizes slice-regular functions whose $\pv$-product vanishes identically on $\Omega$ and shows the effectiveness of Proposition~\ref{star-product}. Notice that by definition the vanishing of the $\pv$-product is equivalent to the fact that $f$ and $g$ commute with respect to the $*$-product.

\begin{prop}\label{wedge-product}
Given $f,g\in\So(\Omega)$, the function $f\pv  g$ vanishes identically if and only if $f_v$ and $g_v$ are linearly dependent over $\So_\R(\Omega)$, namely there exist $\alpha,\beta\in\So_\R(\Omega)$ such that $\alpha f+\beta g\equiv 0$ and either $\alpha\not\equiv0$ or $\beta\not\equiv0$.
\end{prop}
\begin{proof}
First suppose $f\pv  g\equiv0$.
If $f_v\equiv 0$ the assertion is trivial by setting $1\cdot f_v+0\cdot g_v\equiv 0$. Otherwise, 
without loss of generality, we can suppose $f_1\not\equiv 0$.
The identity
$f\pv  g\equiv0$ is equivalent to 
$(f_2g_3-f_3g_2)i+(f_3g_1-f_1g_3)j+(f_1g_2-f_2g_1)k\equiv 0$.
Thanks to Proposition~\ref{GMP}, this gives 
$$
\begin{cases}
f_2g_3-f_3g_2\equiv 0, \\ f_3g_1-f_1g_3\equiv 0, \\ f_1g_2-f_2g_1\equiv 0.
\end{cases}
$$
Last two equations of the previous system give $f_1g_v-g_1 f_v\equiv 0$; as $f_1\not\equiv 0$, the functions $f_v$ and $g_v$ are  linearly dependent over $\So_\R(\Omega)$.

Vice versa, if $f_v$ and $g_v$ are  linearly dependent over $\So_\R(\Omega)$, we can
suppose, up to a rearrangment, that there exist $\alpha,\beta\in\So_{\R}(\Omega)$,
with $\alpha\not\equiv 0$, such that $\alpha f_{v}+\beta g_{v}\equiv 0$.
By taking the $\pv$-product with $g_{v}$ we obtain
$$
0\equiv \alpha f_{v}\pv g_{v}+\beta g_{v}\pv g_{v}= \alpha f_{v}\pv g_{v}.
$$
Since $\alpha\in\So_{\R}(\Omega)$ is not identically zero, then its
zero set is a disjoint union of isolated real zeros and isolated spheres with
real center. Therefore the equality $\alpha f_{v}\pv g_{v}\equiv 0$ implies that
$f_{v}\pv g_{v}\equiv 0$ since the product of $\alpha$ and $f_v\pv g_v$ is the pointwise product.
\end{proof}

\section{Square roots of slice-preserving regular functions}

Since the symmetrized of any slice-regular function always belongs to $\So_\R(\Omega)$ and the symmetrized of a 
slice-preserving function coincides with its square, it is natural to ask when an element in $\So_\R(\Omega)$ has a square root in $\So_\R(\Omega)$ and, in particular, when the symmetrized function of a slice-regular function has a square root in $\So_\R(\Omega)$.  The following results completely solve the question, giving necessary and sufficient conditions on a non-zero slice-preserving function in order to be the square of a slice-preserving function; in particular they apply when we look for the square root of the symmetrized of a given function.

As it will be seen in Corollary~\ref{true-sine}, the results contained in this section can be applied to obtain a more explicit form for the $*$-exponential of a slice-regular function.

Throught the whole section we assume that  each connected component of
$\Omega_I=\Omega\cap \C_I$ is simply connected for some $I\in\SF$ (and so for all $I\in\SF$, being $\Omega$ a circular domain).

\begin{prop}\label{real-square-root}
Given $h\in \So_\R(\Omega)\setminus\{0\}$ there exists $f\in \So_\R(\Omega)$ such that $f^2=h$ if and only if
\begin{enumerate}[(i)]
\item the zero set of $h$ consists of  real zeroes of even isolated  multiplicity and isolated spheres $\SF_{q}$ with spherical multiplicity multiple of $4$;
\item if $\Omega\cap \R\neq\emptyset$ then $h(\Omega\cap \R)\subseteq [0,+\infty)$.
\end{enumerate}
\end{prop}

\begin{proof}
The necessity of the first condition is due to the fact that functions belonging to $\So_\R(\Omega)\setminus\{0\}$ only have real  isolated zeroes and isolated spherical zeroes with real center (see \cite{G-P}). Squaring a slice-preserving function therefore produces real  isolated zeroes with even multiplicity and  isolated spherical zeroes with real center with multiplicity multiple of $4$, as the spherical multiplicity is always an even number~(see~\cite{G-S-St}, Definition 3.37).
The necessity of the second condition is straightforward.
 
We first perform the proof of the sufficiency of conditions (i) and (ii) in the case when $\Omega$ intersects the real axis.
As $h\not\equiv 0$, we can choose a point $q_0\in\Omega   \cap \R$ with $h(q_0)\neq 0$. Now 
fix $I\in \SF$ and choose a family $\{\Omega_n\}$ of nested domains with compact closure $\overline{\Omega_{n}}$ in $\Omega_I=\Omega\cap \C_I$ such that each $\Omega_n$ is simply connected, symmetric with respect to the real axis and contains $q_0$ (such a family certainly exists thanks to Riemann mapping theorem for domains symmetric with respect to the real axis, see~\cite{Gal-S} Theorem 3.3). 

The restriction $h_{I}$ of $h$ to $\Omega_I$ is a holomorphic function from $\Omega_I$ to $\C_I$ which does not vanish at $q_0$ and, by condition (i), has zeroes of even multiplicity both at real points and at conjugate points outside $\R$. Now denote by $Z_n$ the intersection of the zero set  of $h_I$ with $\Omega_n$: since $h_I(q_0)\neq0$ then $Z_n$ is finite and we can find a polynomial $P_n(z)$ with real coefficients and a holomorphic function $g_n:\Omega_{I}\to \C_I$ such that  
$g_n$ does not vanish on $\overline{\Omega_n}$ and $h_I(z)= P_n^2(z)g_n(z)$ for any $z\in\Omega_I$.
As $g_n$ does not vanish on $\Omega_n$, we can find a holomorphic square root $\gamma_n$ of $g_n$ on $\Omega_n$ such that $P_n(q_0)\gamma_n(q_0)$ is real and positive. Since $P_n(z)$ has real coefficients, condition (ii) implies that $\gamma_n(\Omega_n\cap \R)\subset \R$. 
Then the function $\varphi_n=P_n\cdot\gamma_n$ defined on $\Omega_n$ has the following properties: 
\begin{itemize}
\item it is holomorphic on $\Omega_n$, 
\item it is a square root of $h_I$ on $\Omega_n$, that is $\varphi_n^2(z)=h_I(z)$ for any $z\in\Omega_n$,
\item it maps $\Omega_n\cap \R$  into $\R$,
\item it takes positive value at $q_0$.  
\end{itemize}
It is easily seen that the above properties entail the uniqueness of $\varphi_n$ and therefore $\varphi_n\equiv \varphi_{n+1}$ on $\Omega_n$; hence setting $\varphi(z)=\varphi_n(z)$ for any $z\in\Omega_n$ defines a holomorphic function on $\Omega_I$ which is a square root of $h_I$. Now denote by $f$ the regular extension of $\varphi$ to $\Omega$; it is easily seen that such a function is a square root of $h$. Indeed $f\in \So_\R(\Omega)$ because $f$ coincides with $\varphi$ on $\Omega\cap \R$, moreover $f^2=h$ on $\Omega\cap \R$ which ensures $f^2\equiv h$ on $\Omega$.  

Now suppose that $\Omega   \cap \R=\emptyset$.  
Fix $I\in \SF$ and consider $\Omega_I^+=\Omega\cap \{z\in\C_I\,|\, {\rm Im}(z)>0\}$. As $\Omega_I^{+}$ is simply connected and the set of zeroes of $h_I$ in $\Omega_I$ consists of conjugate points of even multiplicity, 
by reasoning as above, we can find a square root $\varphi^+$ of $h_I$ on $\Omega_I^+$. As the function $h \in\So_\R(\Omega)$, then $h_I$ coincides with its Schwarz reflection on $\Omega_I^-=\Omega_I\setminus\Omega_I^+$, so we can define a holomorphic map $\varphi$ on $\Omega_I$ by extending $\varphi^+$ by Schwarz reflection on $\Omega_I^-$. In this way we obtain a holomorphic square root $\varphi$ of $h_I$ on $\Omega_I$; now denote by $f$ the regular extension of $\varphi$ to $\Omega$; as above it is easily seen that $f\in\So_\R(\Omega)$ is a square root of $h$.  
\end{proof}

The above result can be applied to infer the existence of the square root of the symmetrized of a regular function with suitable zeroes. As previously stated, we assume that each connected component of $\Omega_I$ is simply connected.

\begin{cor}\label{square-root}
Given $g\in \So(\Omega)$ there exists $f\in \So_\R(\Omega)$ such that $f^2=g^s$ if and only if
the zero set of $g$ does not contain non real zeroes of odd isolated multiplicity.
\end{cor}

\begin{proof}
If $g^s\equiv 0$, we can take $f\equiv 0$. Then we are left to study the case when $g^s\not\equiv 0$.
In \cite{G-P} the zero set of such $g$'s is explicitely described as the union of isolated points and isolated spheres $\SF_q$.
As $g^s$ maps $\Omega\cap \R$ in $ [0,+\infty)$, thanks to Proposition~\ref{real-square-root} it is enough to check that the zeroes of the symmetrized function $g^s$ are only real zeroes of even isolated multiplicity and spherical zeroes of multiplicity multiple of $4$. This request exactly means that the zero set of $g$ cannot contain non real zeroes of odd isolated multiplicity. 
\end{proof}

Next example gives an explicit application of the above corollary.

\begin{example}
Set $g(q)=q+i$. Then $g^s(q)=q^2+1$ has a spherical zero given by $\SF$ with multiplicity $2$ and so it is not the square of any slice-preserving  regular function.
\end{example}

\section{The $*$-exponential of a slice-regular function}

We now enter in the discussion of the $*$-exponential. In~\cite{C-S-St-2}, Colombo, Sabadini e Struppa introduced the following definitions, motivated by the natural challenge of functional calculus in the non-commutative setting, see also \cite{C-S-St-1, G-R}.  
  
\begin{definiz}
Given $f\in \So(\Omega)$ we set 
$$
\exp_*(f)=\sum_{n\in\N}\frac{f^{*n}}{n!}; \qquad
\cos_*(f)=\sum_{n\in\N}\frac{(-1)^nf^{*(2n)}}{(2n)!}; \qquad
\sin_*(f)=\sum_{n\in\N}\frac{(-1)^nf^{*(2n+1)}}{(2n+1)!}.
$$
\end{definiz}  
As we have just seen, natural estimates on circular compact subsets
of $\Omega$ show that the series converge uniformly and therefore $\exp_*(f)$, $\cos_*(f)$, $\sin_*(f)$ are well defined and belong to $\So(\Omega)$.

In some special cases, $\exp_*$, $\cos_*$, $\sin_*$ take the usual form of the complex case and we sometimes denote them also dropping the underscore $*$.

\begin{remark}
If $f\in \So_\R(\Omega)$ then $\exp_*(f)$ can be written as $\sum_{n\in\N}\frac{f^{n}}{n!}$ and belongs to $ \So_\R(\Omega)$; the same holds for $\cos_*(f)$ and $\sin_*(f)$. If there exists $J\in\SF$ such that $f\in \So_J(\Omega)$ then $\exp_*(f)$ can be obtained as the regular extension to $\Omega$ of the complex function $e^{f_J}$, where $f_J$ denotes the restriction of $f$ to $\Omega_J=\Omega\cap \C_J$; in this case $\exp_*(f)$ belongs to $ \So_J(\Omega)$. Again, the same holds for $\cos_*(f)$ and $\sin_*(f)$, too. 
\end{remark}

In the complex case, one of the most peculiar features of the exponental is its behaviour with respect to the sum; in the quaternionic case this happens under special conditions. The most natural hypothesis is the commutation of $f$ and $g$ which, due to Proposition~\ref{wedge-product}, is equivalent to the fact that $f_v$ and $g_v$ are linearly dependent over $\So_\R(\Omega)$. At the end of the paper we will prove a more refined result, see Theorem~\ref{super-equality}, containing an unexpected couple of conditions on the functions $f$ and $g$ which are equivalent to the equality $\exp_*(f+g)=\exp_*(f)*\exp_*(g)$.

\begin{prop}\label{lin-dep}
If $f_v$ and $g_v$ are linearly dependent on 
 $\So_\R(\Omega)$, then 
$$
\exp_*(f+g)=\exp_*(f)*\exp_*(g).
$$
In particular the above equality holds if either $f$ or $g$ belong to $\So_\R(\Omega)$.
 \end{prop}
\begin{proof}
It holds:
\begin{align*}
\exp_*(f+g)&\!=\!\sum_{n\in\N}\frac{(f+g)^{*n}}{n!}\!=\!
\sum_{n\in\N}\frac{1}{n!}\sum_{m\leq n} \!\binom{n}{m} f^{*m}* g^{*(n-m)}\!=\!
\sum_{n\in\N}\sum_{m\leq n}\!\frac{1}{m!(n-m)!} f^{*m} *g^{*(n-m)}\\
&=\sum_{m\in\N}\sum_{n\geq m} \frac{f^{*m}}{m!}*\frac{g^{*(n-m)}}{(n-m)!}
=\sum_{m\in\N} \frac{f^{*m}}{m!}*\sum_{n\geq m}\frac{g^{*(n-m)}}{(n-m)!}=
\left(\sum_{m\in\N} \frac{f^{*m}}{m!}\right)*\left(\sum_{\nu\in\N}\frac{g^{*\nu}}{\nu!}\right)\\
&=\exp_*(f)*\exp_*(g),
\end{align*}
where in the second equality we took into account the fact that $f$ and $g$ commute because of Proposition~\ref{wedge-product}.
\end{proof}

As an immediate consequence we obtain the following 
\begin{cor}\label{exp-reale}
Let $f=f_0+f_v\in\So(\Omega)$, then 
$$
\exp_*(f)=\exp_*(f_0)*\exp_*(f_v)=\exp(f_0)\exp_*(f_v).
$$
In particular if $f_v(q_0)=0$ then $\exp_*(f)(q_0)=\exp(f_0)(q_0)$.
\end{cor}

The above corollary allows us to interpretate the $*$-exponential in a cosine-sine form. 

\begin{prop}\label{star-exponential}
Let $f=f_0+f_v\in\So(\Omega)$, then 
\begin{equation}\label{cosine-sine}
\exp_*(f)=\exp_*(f_0)\left(\sum_{m\in\N} \frac{(-1)^m (f_v^s)^m}{(2m)!}+\sum_{m\in\N} \frac{(-1)^m (f_v^s)^m}{(2m+1)!} f_v\right). 
\end{equation}
\end{prop}
\begin{proof}
Thanks to Corollary~\ref{exp-reale} we can perform the computation in the case when $f_0\equiv 0$.

\noindent
Notice that Proposition~\ref{star-product} and Remark~\ref{symmetrized} imply
$(f_v)^{*2}=f_v*f_v=-\langle f_v,f_v\rangle_*=-f_v^s$.
Now 
\begin{align*}
\exp_*(f_v)&=\sum_{n\in\N}\frac{f_v^{*n}}{n!}=
\sum_{n\ \rm even}\frac{f_v^{*n}}{n!}+\sum_{n\ \rm odd}\frac{f_v^{*n}}{n!}=
\sum_{m\in\N}\frac{f_v^{*(2m)}}{(2m)!}+\sum_{m\in\N}\frac{f_v^{*(2m+1)}}{(2m+1)!}\\
&=\sum_{m\in\N}\frac{(f_v^{*2})^{*m}}{(2m)!}+\sum_{m\in\N}\frac{(f_v^{*2})^{*m}}{(2m+1)!}*f_v=
\sum_{m\in\N}\frac{(-f_v^{s})^{*m}}{(2m)!}+\sum_{m\in\N}\frac{(-f_v^{s})^{*m}}{(2m+1)!}*f_v\\
&=\sum_{m\in\N}\frac{(-1)^m(f_v^{s})^{m}}{(2m)!}+\sum_{m\in\N}\frac{(-1)^m(f_v^{s})^{m}}{(2m+1)!}f_v
\end{align*}
because $f_v^{s}$ is slice-preserving and the $*$-product becomes the pointwise product.
\end{proof}

Formula~\ref{star-exponential} allows us to write $\exp_*(f)$ as a local cosine-sine analogous of the complex case. 

\begin{cor}\label{sine-cosine}
Let $f=f_0+f_v\in\So(\Omega)$. If 
\begin{enumerate}
\item $f_v^s(q_0)=0$, then 
 $$\exp_*(f)(q_0)=\exp_*(f_0)(q_0)\left(1+ f_v(q_0)\right); $$
\item $f_v^s(q_0)$ is a positive real, denote by $x_0$ one of the square roots of $f_v^s(q_0)$, then 
 $$\exp_*(f)(q_0)=\exp_*(f_0)(q_0)\left(\cos x_0+\frac{\sin x_0}{x_0} f_v(q_0)\right); $$
\item $f_v^s(q_0)$ is a negative real, denote by $x_0$ one of the square roots of $-f_v^s(q_0)$, then 
 $$\exp_*(f)(q_0)=\exp_*(f_0)(q_0)\left(\cosh x_0+\frac{\sinh x_0}{x_0} f_v(q_0)\right);$$ 
\item $f_v^s(q_0)=\alpha_0+\beta_0 J$  with $\beta_0\neq 0$ denote by $a_0+b_0 J$ one of the square roots of  $f_v^s(q_0)$, then 
$$\exp_*(f)(q_0)=\exp_*(f_0)(q_0)\left(\cos (a_0+b_0 J)+\frac{\sin (a_0+b_0 J)}{a_0+b_0 J} f_v(q_0)\right) .$$
\end{enumerate}
\end{cor}
\begin{proof}
The proof is a trivial application of Formula~\eqref{cosine-sine} since in (1) the two power series sum both up to $1$; in  
$(2)$ we have $f_v^s(q_0)=x_0^2$; in $(3)$ we have $f_v^s(q_0)=-x_0^2$ and in $(4)$ we have $f_v^s(q_0)=(a_0+b_0 J)^2$.
Notice that the expression $\frac{\sin(a_0+b_0 J)}{a_0+b_0 J}f_v(q_0)$  in $(4)$ is well defined because all factors lie in the same $\mathbb{C}_J$.
\end{proof}

If the zero set of $f_v$ does not contain non real isolated zeroes of odd multiplicity, the above result can be made more precise.  Indeed in this case Corollary~\ref{square-root} ensures we can find a regular function in $\So_\R(\Omega)$ whose square is $f_v^{s}$ and which therefore gives a global determination of the square root of $f_v^{s}$; in order to stress the analogy with the complex case we denote it by $\sqrt{f_v^s}$.

\begin{cor}\label{true-sine}
Let $f=f_0+f_v\in\So(\Omega)$. If  each connected component of $\Omega_I$ is simply connected,
$f_v^s$ is not identically zero and the zero set of $f_v$ does not contain non real zeroes of odd isolated multiplicity then 
\begin{equation}\label{true-sine-eq}
\exp_*(f)=\exp_*(f_0)
\left(\cos_*\left(\sqrt{f_v^s}\right)+\frac{\sin_*\left(\sqrt{f_v^s}\right)}{\sqrt{f_v^s}} f_v\right),
\end{equation}
where $f_{v}^{s}$ does not vanish.
\end{cor}  
\begin{proof}
As
$$\sum_{m\in\N}\frac{(-1)^m(f_v^{s})^{m}}{(2m)!}=\sum_{m\in\N}\frac{(-1)^m(\sqrt{f_v^s})^{2m}}{(2m)!}=
\cos_*\left(\sqrt{f_v^s}\right)$$ and $$\sum_{m\in\N}\frac{(-1)^m(f_v^{s})^{m}}{(2m+1)!}=
\sum_{m\in\N}\frac{(-1)^m(\sqrt{f_v^s})^{2m}}{(2m+1)!}=\frac1{\sqrt{f_v^s}}\sum_{m\in\N}\frac{(-1)^m(\sqrt{f_v^s})^{2m+1}}{(2m+1)!}=\frac{\sin_*\left(\sqrt{f_v^s}\right)}{\sqrt{f_v^s}},$$
the assertion follows from~\eqref{cosine-sine}.
\end{proof}

\begin{remark}
In the hypothesis of Corollary~\ref{true-sine}, at points where $f_v^s$ does not vanish, we can write Formula~\eqref{true-sine-eq} in the form
\begin{equation*}
\exp_*(f)=\exp_*(f_0)
\left(\cos_*\left(\sqrt{f_v^s}\right)+\sin_*\left(\sqrt{f_v^s}\right)\frac{f_v}{\sqrt{f_v^s}}\right).
\end{equation*}
We underline that $\left(\frac{f_v}{\sqrt{f_v^s}}\right)^s\equiv 1$ outside the zero set of $f_v^s$ and therefore the function $\frac{f_v}{\sqrt{f_v^s}}$ imitates
the behaviour of the imaginary unit in $\C$.
\end{remark}

As a further consequence of Proposition~\ref{star-exponential} we can describe all functions whose $*$-exponential is either slice-preserving or $\C_J$-preserving for a suitable $J\in\SF$.
To simplify computations we set
$$
\mu(f)=\sum_{m\in\N} \frac{(-1)^m (f_v^s)^m}{(2m)!},\qquad
\nu(f)=\sum_{m\in\N} \frac{(-1)^m (f_v^s)^m}{(2m+1)!}
$$
which are slice-preserving regular functions. Then~\eqref{cosine-sine}
 can be reformulated as 
\begin{equation}\label{mu-nu}
\exp_*(f)=\exp_*(f_0)\left(\mu(f)+\nu(f) f_v\right). 
\end{equation}

Notice that, thanks to Corollary~\ref{sine-cosine}, the equality $\nu(f)(q_0)=0$ holds if and only if $f_v^s(q_0)$ belongs to $\mathcal{Z}=\{n^{2}\pi^{2}\,|\,n\in\Z\setminus\{0\}\}$, that is there exists $n\in\Z\setminus\{0\}$ such that $f_v^s(q_0)=n^2\pi^2$. Notice that at such points 
$\mu(f)(q_0)$ is equal to $(-1)^n$.

Formula~\eqref{mu-nu} enables us to characterize when the $*$-exponential of a regular function preserves one or all slices. Next example shows that things are not as simple as one could imagine.

\begin{example}
If $f(q)=\pi\cos(q)i+\pi\sin(q)j$, then $f_1(q)=\pi\cos(q)$,  $f_2(q)=\pi\sin(q)$ and therefore $f_v^s(q)=f_1^2(q)+f_2^2(q)=\pi^2\cos^2(q)+\pi^2\sin^2(q)=\pi^2$ for all $q\in\HH$. By Corollary~\ref{true-sine} 
this implies that $\exp_*(f)=\cos(\pi)=-1$. It is easy to check that the function $f$ does not preserve any slice but its $*$-exponential trivially preserves all slices. 
\end{example}

\begin{prop}
Let $f\in\So(\Omega)$. Then 
\begin{enumerate}[(i)]
\item $\exp_*(f)\in\So_\R(\Omega)$ if and only if either $f_v\equiv 0$ or
$f_v^s$ is a constant which belongs to $\mathcal{Z}$,
\item  there exists  $J\in\SF$ such that $\exp_*(f)\in\So_J(\Omega)\setminus \So_\R(\Omega)$  if and only if $f_v\in \So_J(\Omega)$ and it is not a constant which belongs to $\Z\pi J$.
\end{enumerate}
\end{prop}
\begin{proof} $(i)$
If $f_v\equiv 0$ then $f\in \So_\R(\Omega)$ and therefore $\exp_*(f) \in \So_\R(\Omega)$.
If $f_v\not\equiv 0$ and $f_v^s$ is a constant which belongs to $\mathcal{Z}$, let us denote by $n\neq 0$ an integer such that $f_v^s\equiv n^2\pi^2$;
then Corollary~\ref{true-sine} 
gives $\exp_*(f)=\exp_*(f_0)
\left(\cos_*\left(n \pi\right)+\frac{\sin_*\left(n\pi\right)}{n\pi} f_v\right) 
=\exp_*(f_0)\cos(n\pi)=(-1)^n\exp_*(f_0)\in\So_\R(\Omega)$.

Now if $\exp_*(f)\in\So_\R(\Omega)$, then the function $\exp_*(f_0)\left(\mu(f)+\nu(f) f_v\right)$ does; since $\exp_*(f_0)$  is never-vanishing, this entails that $\mu(f)+\nu(f) f_v$ belongs to $\So_\R(\Omega)$. 
This fact implies that $\nu(f) f_v$ is identically zero. If $\nu(f)\equiv 0$, then 
$\sum_{m\in\N} \frac{(-1)^m (f_v^s)^m}{(2m+1)!}\equiv 0$ and therefore $f_v^s$ belongs to one of the zeroes of the power series which are given by $\mathcal{Z}=\{n^2\pi^2\, |\, n\in\Z\setminus\{0\}\}$ and hence $f_v^s$ is a constant which belongs to $\mathcal{Z}$. If $\nu(f)\not\equiv0$, we have that  
$f_v=0$  outside the zero set of $\nu(f)$ which is closed and has empty interior, therefore by continuity $f_v\equiv 0$, that is $f\in\So_\R(\Omega)$. 

$(ii)$ If $f_v\in \So_J(\Omega)$ and it is not a constant which belongs to $\Z\pi J$, then we can write $f=f_0+f_1J$ where $f_1\not\equiv n\pi$ for some $n\in\Z$. As $f_v^s=f_1^2$, by Formula~\eqref{cosine-sine} we obtain that 
$\exp_*(f)=\exp_*(f_0)\left(\cos_*(f_1)+\sin_*(f_1)J\right)$. As  $\exp_*(f_0)$ is never-vanishing and $f_1$ is not a constant which belongs to $\Z\pi$, then $\exp_*(f)$ belongs to $\So_J(\Omega)\setminus \So_\R(\Omega)$.

Now, if $\exp_*(f)\in\So_J(\Omega)\setminus \So_\R(\Omega)$, then the function $\exp_*(f_0)\left(\mu(f)+\nu(f) f_v\right)$ does; since $\exp_*(f_0)$  is never-vanishing, this entails that $\mu(f)+\nu(f) f_v$ belongs to $\So_J(\Omega)\setminus \So_\R(\Omega)$. 
This fact implies that $\nu(f) f_v$ is equal to $g_1J$ for some $g_1\in \So_\R(\Omega)\setminus\{0\}$. Since $\nu(f)
\in \So_\R(\Omega)\setminus\{0\}$ we obtain that $f_v\in \So_J(\Omega)\setminus\{0\}$. Would $f_v\equiv n\pi J$ for some $n\in\Z\setminus\{0\}$, then $f_v^s\equiv n^2\pi^2$ and therefore $\nu(f)\equiv 0$, which is a contradiction.
\end{proof}

Formula~\eqref{mu-nu} can also be used to obtain an explicit expression for 
$(\exp_*(f))*(\exp_*(g))$.

\begin{prop}\label{star-product-exponential}
Let $f,g\in \So(\Omega)$. Then $\exp_*(f)*\exp_*(g)$ is equal to
$$\exp_*(f_0)\exp_*(g_0)\left(\mu(f)\mu(g)-\nu(f)\nu(g)\langle f_v,g_v\rangle_*+\nu(f)\nu(g) f_v\pv g_v
+\mu(f)\nu(g) g_v+\mu(g)\nu(f)f_v\right).
$$
\end{prop}

\begin{proof}
The assumption is a direct offspring of Formulas~\eqref{mu-nu} and~\eqref{star-formula}.  
\end{proof}

Since $f^c=f_0-f_v$, Formula~\eqref{mu-nu} also gives $\exp_*(f^c)$ in terms of $\exp_*(f)$. Indeed we have 

\begin{remark}
Let $f\in\So(\Omega)$. Then
$\exp_*(f^c)=\left(\exp_*(f)\right)^c$.
\end{remark}

When applied to $f$ and $f^c$ the above Proposition shows that, as it is natural to expect, the $*$-exponential of a regular function never vanishes. Indeed we have the following

\begin{prop}\label{real-part}
Let $f\in\So(\Omega)$. Then
\begin{align}
(\exp_*(f))^s&=\exp_*(2f_0) \label{eq1},\\
\frac{\exp_{*}(f)+\exp_{*}(f^{c})}{2}&=\exp_{*}(f_0)\mu(f),\label{eq2}\\
\frac{\exp_{*}(f)-\exp_{*}(f^{c})}{2}&=\exp_{*}(f_0)\nu(f)f_{v},\label{eq3}\\
\exp_{*}(f)*\exp_{*}(-f)&\equiv 1.\label{eq4}
\end{align}
In particular $\exp_*(f)$ is a never-vanishing function.
\end{prop}

\begin{proof}
Indeed we have
\begin{align*}
(\exp_*(f))^s&=(\exp_*(f))*(\exp_*(f))^c=(\exp_*(f))*(\exp_*(f^c))\\
&=\exp_*(f_0)\exp_*(f_0)\left(\mu(f)\mu(f)+\nu(f)\nu(f)\langle f_v,f_v\rangle_*-\mu(f)\nu(f) f_v+\mu(f)\nu(f)f_v\right)\\&=
\exp_*(2f_0)\left(\mu(f)\mu(f)+\nu(f)\nu(f) f_v^s\right)=
\exp_*(2f_0)\left(\mu(f)^2+\nu(f)^2 f_v^s\right)
\end{align*}
Now using the same notation as in Corollary~\ref{sine-cosine} we have that $\mu(f)^2+\nu(f)^2 f_v^s\equiv1$ and therefore we obtain Equality~\eqref{eq1}.
As $\exp_*(2f_0)=\exp(2f_0)$ is never-vanishing, then also $\exp_*(f)$ has no zeroes on $\Omega$.
Equalities~\eqref{eq2} and~\eqref{eq3} follow immediately from~\eqref{mu-nu}, while~\eqref{eq4} can be obtained from Proposition~\ref{star-product-exponential} following 
the same kind of computations as above.
\end{proof}

Proposition~\ref{star-product-exponential} enables us to deepen our understanding of the comparison between the $*$-product of $\exp_*(f)$ and $\exp_*(g)$ and $\exp_*(f+g)$, thus giving, when $\Omega$ contains real points, a necessary and sufficient condition for them to be equal. We underline that, while Condition (i) appears quite natural, Condition (ii) is in some sense surprising: in particular it holds for a class of functions larger than one could initially presume (see  Example~\ref{non-constant}).

%

\begin{teo}\label{super-equality}
 Let $\Omega$ contain real points. Take $f,g\in\So(\Omega)$. 
If 
\begin{equation}\label{equality}
\exp_*(f+g)=
\exp_*(f)*\exp_*(g)
\end{equation}
then either 
\begin{enumerate}[(i)]
\item $f_v$ and $g_v$ are linearly dependent over $\So_\R(\Omega)$ or 
\item there exist $n,m,p\in\mathbb Z\setminus\{0\}$ such that $f_v^s\equiv n^2\pi^2$,  $g_v^s\equiv m^2\pi^2$, $2\langle f_v, g_v \rangle_*=(p^2 -n^2 -m^2)\pi^2$ and $n+m\equiv p$ mod 2.  
\end{enumerate}
Vice versa, if either $(i)$ or $(ii)$ are satisfied,
then~\eqref{equality} holds.
\end{teo}

\begin{proof}
By Formula~\eqref{mu-nu}  and Proposition~\ref{star-product-exponential} 
we can suppose that $f_0$ and $g_0$ vanish identically. 
Then we have 
\begin{align*}
\exp_*(f+g)&=\mu(f+g)+\nu(f+g)(f+g)_v\\
\exp_*(f)*\exp_*(g)&=\mu(f)\mu(g)-\nu(f)\nu(g)\langle f_v,g_v\rangle_*\\
&\qquad\qquad\quad+\nu(f)\nu(g) f_v\pv g_v
+\mu(f)\nu(g) g_v+\mu(g)\nu(f)f_v.
\end{align*}

First we study the case when neither $\nu(f)$ nor $\nu(g)$ are identically zero.
Since $(f+g)_v=f_v+g_v$ belongs to ${\rm Span}_{\So_\R(\Omega)}(f_v,g_v)$,  if   $\exp_*(f+g)=\exp_*(f)*\exp_*(g)$ then 
$\nu(f)\nu(g) f_v\pv g_v$ belongs to ${\rm Span}_{\So_\R(\Omega)}(f_v,g_v)$ too. 
Then there exist $\alpha,\beta\in\So_\R(\Omega)$ such that 
$\nu(f)\nu(g)f_v\pv g_v=\alpha f_v+\beta g_v$. If $\alpha$ and $\beta$ are both identically zero,
then $\nu(f)\nu(g)f_v\pv g_v\equiv 0$.
Thus $f_v\pv g_v\equiv 0$ and Proposition~\ref{wedge-product} ends the proof. If $\alpha$ and $\beta$
are not both identically zero, up to rearrangement, we can suppose $\alpha\not\equiv 0$.
Now $0\equiv\langle\nu(f)\nu(g)f_v\pv g_v ,f_v\rangle_*=\alpha f_v^s+\beta \langle f_v,g_v\rangle_*$
and $0\equiv\langle \nu(f)\nu(g)f_v\pv g_v,g_v\rangle_*=\alpha \langle f_v,g_v\rangle_*+\beta g_v^s$ that is 
$$
\begin{cases}
\alpha f_v^s+\beta \langle f_v,g_v\rangle_*\equiv0,\\
\alpha \langle f_v,g_v\rangle_*+\beta g_v^s\equiv0.
\end{cases}
$$
Taking a suitable combination of these equations we obtain $\alpha\left(f_v^sg_v^s-\langle f_v, g_v\rangle_*^2\right)\equiv 0$. 
As $\alpha\in\So_{\R}(\Omega)\setminus\{0\}$, we obtain that $f_v^sg_v^s-\langle f_v, g_v\rangle_*^2=(f_v\pv g_v)^s$ is identically zero. Since $\Omega$ contains real points, then $\So(\Omega)$ is an integral domain, and therefore $f_v\pv g_v$ is identically zero. Again, Proposition~\ref{wedge-product} entails that $f_v$ and $g_v$ are linearly dependent over $\So_\R(\Omega)$.

If $\nu(f)\equiv0$ (the case $\nu(g)\equiv 0$ is treated analogously), then $f_v^s$ is identically equal to $n^2\pi^2$ for a suitable $n\in\Z\setminus\{0\}$. 
In this case $\exp_*(f+g)=
\exp_*(f)*\exp_*(g)
$ gives 
$$\mu(f+g)+\nu(f+g)(f+g)_v=\mu(f)\mu(g)+\mu(f)\nu(g) g_v$$ 
which, taking the equality of the vector parts, implies  
$$\nu(f+g)(f_v+g_v)=\mu(f)\nu(g) g_v,$$ 
that can be written as 
$$\left(\mu(f)\nu(g)-\nu(f+g)\right)g_v=\nu(f+g)f_v.$$
This gives the linear dependency of $f_v$ and $g_v$ over $\So_\R(\Omega)$
 unless  $\mu(f)\nu(g)-\nu(f+g)\equiv\nu(f+g)\equiv0$. In this last case, we obtain
$\mu(f)\nu(g)\equiv\nu(f+g)\equiv0$. A simple computation gives $\mu(f)=(-1)^n$ and therefore $\nu(g)\equiv\nu(f+g)\equiv0$. 
Then there exist $m,p\in\mathbb{Z}\setminus\{0\}$ such that $g_v^s\equiv m^2\pi^2$ and $(f_v+g_v)^s= p^2\pi^2$ and therefore $\mu(f+g)=(-1)^p$. 

Since $(f_v+g_v)^s=-(f_v+g_v)*(f_v+g_v)=n^2\pi^2+m^2\pi^2+2\langle f_v,g_v\rangle_*$ we obtain $2\langle f_v, g_v \rangle_*=(p^2 -m^2 -n^2)\pi^2$. 

Under these assumptions we have 
$$
\exp_*(f+g)=\mu(f+g)\equiv(-1)^p\quad\mbox{and}\quad
\exp_*(f)*\exp_*(g)=\mu(f)\mu(g)\equiv(-1)^n(-1)^m=(-1)^{n+m}
$$
which holds if and only if $n+m$ and $p$ have the same parity.

If $(i)$ is satisfied, then~\eqref{equality} holds because of Proposition~\ref{lin-dep}. If $(ii)$ is satisfied, then $\nu(f)=\nu(g)=\nu(f+g)\equiv0$ and $\mu(f)=(-1)^n$, $\mu(g)=(-1)^m$, $\mu(f+g)=(-1)^p$ so that 
$\exp_*(f+g)=\mu(f+g)=(-1)^p$ and  $\exp_*(f)*\exp_*(g)=(-1)^n(-1)^m=(-1)^{n+m}=(-1)^{p}$ because $n+m$ and $p$ have the same parity.
\end{proof}

The following two examples show cases where~\eqref{equality} holds while $f_v$ and $g_v$ are not linearly dependent over $\So_\R(\HH)$.  In particular the second one gives a large class of non constant functions which do not commute but satisfy~\eqref{equality}.

\begin{example}\label{constant_orthogonal}
If $I,J\in\SF$ are orthogonal, $f_v\equiv 3\pi I$ and $g_v\equiv 4\pi J$ then 
$$\exp_*(f+g)=-\exp_*(f_0+g_0)=\exp_*(f)*\exp_*(g).$$
\end{example}

\begin{example}\label{non-constant}
Let $n,p,m\in\N$ such that $m^2=n^2+p^2$.
Now choose $\alpha,\beta\in\So_\R(\Omega)$ and set 
\begin{align*}
f(q)&=-n\pi\left(\cos(\alpha(q))\cos(\beta(q)) i+\cos(\alpha(q))\sin(\beta(q)) j+\sin(\alpha(q))k\right),\\
g(q)&=\pi((n\cos(\alpha(q))\cos(\beta(q))+p\sin(\alpha(q))\cos(\beta(q))) i+\\
&+(n\cos(\alpha(q))\sin(\beta(q))+p\sin(\alpha(q))\sin(\beta(q))) j\\
&+(n\sin(\alpha(q))-p\cos(\alpha(q)))k).
\end{align*}
As $\alpha,\beta$ are slice-preserving functions then $f,g\in\So(\Omega)$, moreover 
$f_v^s=n^2\pi^2$, $g_v^s=m^2\pi^2$ and $\langle f_v, g_v\rangle_*=-n^2\pi^2$. Since $n,p,m$ is a Pythagorean triple then $2\langle f_v, g_v\rangle_*=-2n^2\pi^2=\pi^2(-n^2-p^2-n^2+p^2)=\pi^2(p^{2}-n^{2}-m^{2})$ and hence the given functions satisfy
the requirements of the statement because the parity condition is satisfied for any Pythagorean triple.
%
%
%
\end{example}

Next  example gives an explicit couple of functions $f$ and $g$ where the parity relation between $n,m$ and $p$ does not hold.

\begin{example}\label{constant_non_orthogonal}
If $I,J\in\SF$ satisfy $IJ+JI=-\frac{13}{20}$, $f_v\equiv 2\pi I$ and $g_v\equiv 5\pi J$ then 
$n=2$, $m=5$ and $p=4$ since
$((f+g)_v)^s=(2^2+5^2+2\cdot 2\cdot 5(IJ+JI))\pi^2=4^2\pi^2$. Then we have $\mu(f)=\mu(f+g)\equiv1$, $\mu(g)\equiv-1$ and $\nu(f)=\nu(g)=\nu(f+g)\equiv0$ which entail
$\exp_*(f+g)=\exp_*(f_0+g_0)$ while $\exp_*(f)*\exp_*(g)= -\exp_*(f_0+g_0)$.
\end{example}

The example below illustrates a case where $\Omega$ does not contain real points.

\begin{example}\label{constant_non_real}
Denote by $\tau:\HH\setminus \R\to\SF$ be given by $\tau(\alpha+\beta I)=\begin{cases} I &\text{if } \beta>0,\\ 
-I &\text{if } \beta<0. 
\end{cases}$
Consider $f,g:\HH\setminus \R\to\HH$ defined as 
\begin{align*}
f(\alpha+\beta I)&=2\pi i-2\pi(1+\tau(\alpha+\beta I)i)j=2\pi(i-j-\tau(\alpha+\beta I)k) \\
g(\alpha+\beta I)&=\tau(\alpha+\beta I)i+\pi j+k.
\end{align*}
Now we have $(f+g)(\alpha+\beta I)=(2\pi+\tau(\alpha+\beta I))i-\pi j+(-2\pi\tau(\alpha+\beta I)+1)k$, $f_v^s=4\pi^2$, $g_v^s=\pi^2$ and $((f+g)_v)^s=(2\pi+\tau(\alpha+\beta I))^2+\pi ^2+(-2\pi\tau(\alpha+\beta I)+1)^2=\pi^2$.

This entails $\mu(f)\equiv1$, $\mu(g)\equiv-1$, $\mu(f+g)\equiv-1$, $\nu(f)\equiv\nu(g)\equiv \nu(f+g)\equiv0$
and therefore
$\exp_*(f+g)\equiv-1\equiv\exp_*(f)*\exp_*(g)$ holds.
\end{example}



\bibliographystyle{amsplain}

\end{document}